\documentclass[11pt,twoside]{article}

\usepackage{mathrsfs,amsfonts,amsmath,amssymb}
\usepackage{latexsym}
\newtheorem{satz}{Theorem}

\newtheorem{theorem}[satz]{Theorem}
\newtheorem{lemma}[satz]{Lemma}

\newtheorem{corollary}[satz]{Corollary}
\newtheorem{remark}[satz]{Remark}
\newtheorem{example}[satz]{Example}

\def\eps{\varepsilon}
\def\_phi{\varphi}

\def\a{\alpha}
\def\d{\delta}
\def\la{\lambda}
\def\ind{{\rm ind}}
\def\v{\vec}
\def\F{{\mathbb F}}

\def\t{\tilde}
\def\o{\omega}

\def\C{{\mathbb C}}
\def\R{{\mathbb R}}
\def\E{\mathsf {E}}
\def\T{{\mathbb T}}

\def\Z_N{{\mathbb Z}_N}
\def\Z{{\mathbb Z}}

\def\Gr{{\mathbf G}}

\def\D{{\mathbb D}}

\def\tr{{\rm tr\,}}

\def\FF{\widehat}

\def\D{\Delta}
\def\Cf{{\mathcal C}}

\def\T{\mathsf {T}}

\def\P{{\cal P}}
\def\L{{\cal L}}
\def\I{{\cal I}}

\def\SL{{\rm SL\,}}
\def\Aff{{\rm Aff\,}}
\def\Stab{{\rm Stab\,}}

\def\HH{{\mathcal H}}

\author{Shkredov I.D.}
\title{Some remarks on products of sets in the Heisenberg group and in the affine group 
	\footnote{This work is supported by the Russian Science Foundation under grant 19--11--00001.}
}
\date{}
\begin{document}
	\maketitle

\begin{center}
	Annotation.
\end{center}

{\it \small
	We obtain some new results on products of large and small sets in the  Heisenberg group as well as in the affine group over the prime field. 
	Also, we derive an application of these growth results to Freiman's isomorphism in nonabelian groups. 	
}
\\

	\section{Introduction}
	\label{sec:introduction}

	Let $p$ be an odd prime number, and $\F_p$ be the finite field.
	Given two sets $A,B\subset \F_p$, define the  \textit{sumset}, the \textit{product set} and the \textit{quotient set} of $A$ and $B$ as 
	$$A+B:=\{a+b ~:~ a\in{A},\,b\in{B}\}\,,$$
	$$AB:=\{ab ~:~ a\in{A},\,b\in{B}\}\,,$$
	and 
	$$A/B:=\{a/b ~:~ a\in{A},\,b\in{B},\,b\neq0\}\,,$$
	correspondingly.
	This paper is devoted to the so--called {\it sum--product phenomenon},
	which says that either the sumset or the product set of a set must be large up to some natural algebraic constrains.  
	One of 
	the strongest form of this principle is the Erd\H{o}s--Szemer\'{e}di  conjecture \cite{ES},
	which says that for any sufficiently large set $A$ of real numbers and an arbitrary $\epsilon>0$ one has
	\begin{equation*}\label{f:sum-product_ES}
	\max{\{|A+A|,|AA|\}}\gg{|A|^{2-\epsilon}} \,.
	\end{equation*}
	The best up to date results in the direction 
	can be found in \cite{Shakan}
	and in 
	\cite{RSS} 
	for  $\R$ and $\F_p$, respectively. 
	Basically, in this 
	paper we restrict ourselves to the case of the finite fields only.

	It is well--known that the sum--product phenomenon is connected with growth in the group of affine transformations, see, e.g., \cite{Brendan_rich}, \cite{RS_SL2}.
	Another group which is connected to this area is the Heisenberg group $\HH$ of $3\times 3$ unipotent matrices and this case was considered in papers  \cite{HH_isom}---\cite{HH_sum-prod_HH} as well as in a more general context, see \cite{BGreenII} and \cite{HaH}, say.  
	For example, in \cite{HH_sum-prod_HH} the following result was obtained.
	
	\begin{theorem}
		Let $A\subset \R$ be a set and 
		\[
		\HH \supseteq \mathcal{A} := \left\{
		\left( {\begin{array}{ccc}
			1 & a & 0\\
			0 & 1 & b \\
			0 & 0 & 1 \\
			\end{array} } \right) ~:~ a,b \in A  \right\} \,.
		\]
		Then for any $\eps>0$ one has 
		\begin{equation}\label{f:HH_sum-prod}
		|\mathcal{A} \mathcal{A}| \ge |A|^2 \max \{ |AA|, |A+A| \} \gg_\eps |A|^{3+7/22 - \eps} \,.  
		\end{equation}
		\label{t:HH_sum-prod}
	\end{theorem}

	Thus formula \eqref{f:HH_sum-prod} shows that the products in $\HH$ are directly connected with the sum--product quantities  $AA$ and $A+A$ 
	similar 
	as the products of sets in the affine group.
	Nevertheless, in a certain sense the affine group is more correlates with the multiplication and the Heisenberg group correlates with the addition, see the discussion of trivial representations in Section \ref{sec:large_sets}.

	We improve Theorem \ref{t:HH_sum-prod} and, moreover, generalize it for so--called {\it bricks}, see Theorem \ref{t:energy_bricks_Heisenberg} in Section \ref{sec:proof}.

	\begin{theorem}
		Under the same conditions as in Theorem \ref{t:HH_sum-prod} one has 
		\[
		|\mathcal{A} \mathcal{A}| \gg |\mathcal{A}|^{7/4+c} \,,
		\]
		where $c>0$ is an absolute constant. 
		Moreover, if $A\subseteq \F_p$, then 
		\begin{equation}\label{f:energy_bricks_Heisenberg_HH_intr}
		|\mathcal{A} \mathcal{A}| \gg \min\{ |\mathcal{A}|^{7/4}, p |\mathcal{A}| \} \,.
		\end{equation}
		\label{t:H_sum-prod_intr}
	\end{theorem}

	It was conjectured in \cite{HH_sum-prod_HH} that, actually, the right exponent in \eqref{f:HH_sum-prod} is four and we have obtained $7/2+c$ in $\R$.

	Using the representation theory and the incidences theory in $\F_p$, we have found new bounds for products of large subsets from the Heisenberg group as well from the affine group, see Theorem \ref{t:z_Heisenberg} and Corollary \ref{c:z_Aff} below. 
	Also, we improve the dependence of $n$ on $\eps$ as well as the dependence on $|Z|$  in the following result from \cite[Theorem 1.3]{HH_Hn} (see Theorem \ref{t:bricks_Hn} from Section \ref{sec:proof}).

	\begin{theorem}
		Let $\eps>0$. 
		Then there exists $n_0 (\eps)$ such that for all $n\ge n_0 (\eps)$ and any sets  $X_i,Y_i, Z\subseteq \F_p$, $i\in [n]$,  
		$X = \prod_{i=1}^n X_i \subseteq \F^n_p$, $Y = \prod_{i=1}^n Y_i \subseteq \F^n_p$ if we form
		\[\mathcal{A} = \{ [x,y,z] ~:~ x\in X,\, y\in Y,\, z\in Z \} \subseteq \HH_n\] 
		with 
		\begin{equation}\label{f:HH_bricks_Hn}
		|\mathcal{A}| > |\HH_n|^{3/4+\eps} \,,
		\end{equation}
		then $\mathcal{A}^2$ contains at least $|\mathcal{A}|/p$ cosets of $[0,0,\F_p]$. 
		\label{t:HH_bricks_Hn}
	\end{theorem}

	In \cite{HH_isom} it was found an interesting application  of products of sets in the Heisenberg group to so--called models of Freiman isomorphisms.
	It was showed that there is a (nonabelian) group, namely, the Heisenberg group such that any set with the doubling constant less than two does not has any good model, see \cite[Section 5.3]{TV}.
	Recall the required definitions and formulate our result.

	Let $\Gr_1$, $\Gr_2$ be groups, $A \subseteq \Gr_1$, $B \subseteq \Gr_2$
	and $s\ge 2$ be a positive integer. 
	A map $\rho : A \to B$  is said to be a {\it Freiman $s$--homomorphism} if for all $2s$--tuples   $(a_1,\dots, a_s,b_1,\dots, b_s) \in A^s \times B^s$ and any signs $\eps_j \in \{-1, 1 \}$, we have 
	\[
	a^{\eps_1}_1 \dots a^{\eps_s}_s = b^{\eps_1}_1 \dots b^{\eps_s}_s 
	\implies
	\rho(a_1)^{\eps_1} \dots \rho(a_s)^{\eps_s} = \rho(b_1)^{\eps_1} \dots \rho(b_s)^{\eps_s}  \,. 
	\]
	If moreover $\rho$ is bijective and $\rho^{-1}$ is also a Freiman $s$--homomorphism, then $\rho$ is called a {\it Freiman $s$--isomorphism}. 
	In this case $A$ and $B$ are said to be Freiman $s$--isomorphic.

	\begin{theorem}
		Let $n$ be a positive integer and $\eps \in (0,1/6)$ be any real number. 
		Then there is a finite (nonabelian) group $H$ and a set $A_* \subset H$ with the following properties: \\
		$1)~$ $|A_*| > n$, $|A_* A_*| < 2|A_*|$;\\
		$2)~$ For any $A\subseteq A_*$, $|A| \ge |A_*|^{1-\eps}$ and any finite group $\Gr$ such that   there exists a Freiman $5$--isomorphism from $A$ to $\Gr$, we have $|\Gr| \gg |A|^{1+\frac{1-6\eps}{5}}$. 
		\label{t:HH_isom_intr}
	\end{theorem}

	It is well--known \cite[Proposition 1.2]{GR} that in abelian case the situation above is not possible and Theorem \ref{t:HH_isom_intr} shows that the picture changes drastically already in the simplest nonabelian case of a two--step nilpotent group. 
	Previously, in \cite{HH_isom} the authors proved 
	an analogue of Theorem \ref{t:HH_isom_intr}   
	for $6$--isomorphisms (our arguments follow their scheme but are slightly simpler). 
	It is easy to see from our proof that, although,  the constant $5$ possibly can be improved but it is the limit of the method.


	All logarithms are to base $2.$ The signs $\ll$ and $\gg$ are the usual Vinogradov symbols.
	For a positive integer $n,$ we set $[n]=\{1,\ldots,n\}.$
	Having a set $A$,  we will write $a \lesssim b$ or $b \gtrsim a$ if $a = O(b \cdot \log^c |A|)$, $c>0$.

	The author is grateful to Misha Rudnev for
	useful discussions. 

\section{Notation}
\label{sec:definitions}

In this paper $\Gr$ is a group with the identity element $e$, $\F$ is a field,  $\F^* = \F \setminus \{0\}$, and $p$ is an odd prime number, 
$\F_p = \Z/p\Z$. 
%
%
Also, we use the same letter to denote a set $A\subseteq \F$ and  its characteristic function $A: \F \to \{0,1 \}$.  

Put
$\E^{+}(A,B)$ for the {\it common additive energy} of two sets $A,B \subseteq \F$
(see, e.g., \cite{TV}), that is, 
$$
\E^{+} (A,B) = |\{ (a_1,a_2,b_1,b_2) \in A\times A \times B \times B ~:~ a_1+b_1 = a_2+b_2 \}| \,.
$$
If $A=B$, then  we simply write $\E^{+} (A)$ instead of $\E^{+} (A,A)$
and the quantity $\E^{+} (A)$ is called the {\it additive energy} in this case. 
One can consider $\E^{+}(f)$ for any complex function $f$ as well.  
More generally, 
we deal with 
a higher energy
\begin{equation}\label{def:T_k_ab}
\T^{+}_k (A) := |\{ (a_1,\dots,a_k,a'_1,\dots,a'_k) \in A^{2k} ~:~ a_1 + \dots + a_k = a'_1 + \dots + a'_k \}|
\,.
\end{equation}
Sometimes we  use representation function notations like $r_{AB} (x)$ or $r_{A+B} (x)$, which counts the number of ways $x \in \F$ can be expressed as a product $ab$ or a sum $a+b$ with $a\in A$, $b\in B$, respectively. 
Further clearly
\begin{equation*}\label{f:energy_convolution}
\E^{+} (A,B) = \sum_x r_{A+B}^2 (x) = \sum_x r^2_{A-B} (x) = \sum_x r_{A-A} (x) r_{B-B} (x) \,.
\end{equation*}
Similarly, one can define $\E^\times (A,B)$, $\E^{\times} (A)$, $\E^{\times} (f)$  and so on.  
In nonabelian setting the energy of a set $A, B \subseteq \Gr$ is (see \cite{sh_hyp}) 
\[
\E^{} (A,B) = |\{ (a_1,a_2,b_1,b_2) \in A\times A \times B \times B ~:~ a_1 b^{-1}_1 = a_2 b^{-1}_2 \}| \,.
\]
Clearly, $\E^{} (A,B) \le |A| |B| \min\{|A|, |B|\}$ and $\E^{} (A,B) \ge |A||B|$.
We write $[A,B]$ for the set of all commutators of $A$ and $B$, namely, $[A,B] = \{ aba^{-1}b^{-1} ~:~ a\in A,\, b\in B\}$.

\bigskip

We finish this section recalling some notions and simple facts from the representations theory, see, e.g., \cite{Serr_representations}.
For a finite group $\Gr$ let $\FF{\Gr}$ be the set of all irreducible unitary representations of $\Gr$. 
It is well--known that size of $\FF{\Gr}$ coincides with  the number of all conjugate classes of $\Gr$.  
For $\pi \in \FF{\Gr}$ denote by $d_\pi$ the dimension of this representation and we write $\langle \cdot, \cdot \rangle_{HS}$ for the correspondent Hilbert--Schmidt scalar product 
$\langle A, B \rangle_{HS}:= \tr (AB^*)$, where $A,B$ are any $(d_\pi \times d_\pi)$--matrices. 
Clearly, $\langle \pi(g) A, \pi(g) B \rangle_{HS} = \langle A, B \rangle_{HS}$.
Also, we have $\sum_{\pi \in \FF{\Gr}} d^2_\pi = |\Gr|$.

For any $f:\Gr \to \C$ and $\pi \in \Gr$ define the matrix $\FF{f} (\pi)$ which is called the Fourier transform of $f$ at $\pi$ by the formula 
\begin{equation}\label{f:Fourier_representations}
\FF{f} (\pi) = \sum_{g\in \Gr} f(g) \pi (g) \,.
\end{equation}
Then the inverse formula takes place
\begin{equation}\label{f:inverse_representations}
f(g) = \frac{1}{|\Gr|} \sum_{\pi \in \FF{\Gr}} d_\pi \langle \FF{f} (\pi), \pi (g^{-1}) \rangle_{HS} \,,
\end{equation}
and the Parseval identity is 
\begin{equation}\label{f:Parseval_representations}
\sum_{g\in \Gr} |f(g)|^2 = \frac{1}{|\Gr|} \sum_{\pi \in \FF{\Gr}} d_\pi \| \FF{f} (\pi) \|^2_{HS} \,.
\end{equation}
The main property of the Fourier transform is the convolution formula 
\begin{equation}\label{f:convolution_representations}
\FF{f*g} (\pi) = \FF{f} (\pi) \FF{g} (\pi) \,,
\end{equation}
where the convolution of two functions $f,g : \Gr \to \C$ is defined as 
\[
(f*g) (x) = \sum_{y\in \Gr} f(y) g(y^{-1}x) \,.
\]
Finally, it is easy to check that for any matrices $A,B$ one has $\| AB\|_{HS} \le \| A\|_{op} \| B\|_{HS}$ and $\| A\|_{op} \le \| A \|_{HS}$, where  the operator $l^2$--norm  $\| A\|_{op}$ is just the absolute value of the maximal eigenvalue of $A$.

\section{Preliminaries}
\label{sec:preliminaries}

Let $\F$ be a field. 
Let $\mathcal{P} \subseteq \F \times \F$ be a set of points  and $\L$ be a collection of lines in $\F \times \F$. 
Having $p\in \mathcal{P}$ and $l \in \L$, we write 
\begin{displaymath}
\I (r,l) = \left\{ \begin{array}{ll}
1 & \textrm{if } p\in l\\
0 & \textrm{otherwise.}
\end{array} \right.
\end{displaymath}
Put $\mathcal{I} (\mathcal{P}, \L) = \sum_{r\in \P, l \in \L} \I (p,l)$.
We will omit to write the conditions $r\in \P$ and $l \in \L$ below.

A trivial upper bound for $\I(\P,\L)$ is 
\begin{equation}\label{f:inc_trivial}
\I (\P,\L) \le \min \{ |\P|^{1/2} |\L|+ |\P|, |\L|^{1/2} |\P| + |\L| \} \,,
\end{equation} 
see, e.g., \cite[Section 8]{TV}.
Further, there is a bound of Vinh \cite{Vinh} (also, see \cite[Section 3]{sh_as}) which says that
\begin{equation}\label{f:Vinh} 
\left| \sum_{r\in l} \I(r,l) f(r) g(l) \right| \le \sqrt{p} \| f\|_2 \|g \|_2 \,,  
\end{equation}  
where either $\sum_r f(r) =0$ or $\sum_{l} g(l) = 0$. 
Finally, a well--known result of Stevens--de Zeeuw
gives us 
an 
asymptotic formula for the number of points/lines incidences in the case when  the set of points forms a Cartesian product,
see \cite{SdZ}, and also \cite{sh_as}.  

\begin{theorem}
	Let $A,B \subseteq \F_p$ be sets, 
	$\mathcal{P} = A\times B$, and $ \mathcal{L}$ be a collection of lines in $\F^2_p$.
	Then 
	\begin{equation}\label{f:line/point_as}
	\mathcal{I}(\mathcal{P}, \mathcal{L}) - \frac{|A| |B| |\mathcal{L}|}{p} \ll |A|^{3/4} |B|^{1/2} |\mathcal{L}|^{3/4} + |\mathcal{L}| + |A| |B| \,.
	\end{equation}
	\label{t:line/point_as}	
\end{theorem}

The proof rests on a well--known points/planes result from \cite{Rudnev_pp} (also, see \cite{sh_as}, \cite{Vinh}).

\begin{theorem}
	Let $p$ be an odd prime, $\mathcal{P} \subseteq \F_p^3$ be a set of points and $\Pi$ be a collection of planes in $\F_p^3$. 
	Suppose that $|\mathcal{P}| \le |\Pi|$ and that $k$ is the maximum number of collinear points in $\mathcal{P}$. 
	Then the number of point--planes incidences satisfies 
	\begin{equation}\label{f:Misha+_a}
	\mathcal{I} (\mathcal{P}, \Pi)  - \frac{|\mathcal{P}| |\Pi|}{p} \ll |\mathcal{P}|^{1/2} |\Pi| + k |\Pi| \,.	
	\end{equation}
	\label{t:Misha+}	
\end{theorem}


\section{On products of large subsets of the affine group and the Heisenberg group}
\label{sec:large_sets}



Let $n\ge 1$ be a positive integer. 
By $\HH_n$ define the Heisenberg linear group over $\F$ consisting of matrices 
\[
[x,y,z] = [\v{x},\v{y},z] := 
\left( {\begin{array}{ccc}
	1 & \v{x} & z \\
	\v{0}^n & \v{1}^n & \v{y} \\
	\v{0}^n & 0 & 1 \\
	\end{array} } \right) \,.
\]
For $n=1$ we write $\HH = \HH_1$. 
The product rule in $\HH_n$ is 
\begin{equation}\label{f:Heisenberg_low}
[x,y,z] \cdot [x',y',z'] = [x+x',y+y',z+z'+xy'] \,,  
\end{equation}
where $xy'$ is the scalar product of vectors $\v{x}$ and $\v{y'}$. 
Also, one has 
\begin{equation}\label{f:Heisenberg_inverse}
[x,y,z]^{-1} = [-x,-y,-z+xy] \,.  
\end{equation}
Clearly, $|\HH_n| = |\F|^{2n+1}$ and there are $|\F|^{2n} + |\F|-1$ conjugate classes of the form $[x,y,0]$, $(x,y) \neq 0$ and $[0,0,z]$, $z\in \F$. 
For any $a,a' \in \HH_n$, $a=[x,y,z]$, $a'=[x',y',z']$, their commutator equals $[a;a'] = [0,0,xy'-yx']$.
Thus the centre of $\HH_n$ is $[0,0,z]$, $z\in \F$ and hence $\HH_n$ is a two--step nilpotent group. 
Given $[x_0,y_0,z_0] \in \HH_n$, we see that the centralizer $\mathrm{C} ([x_0,y_0,z_0]) = \{ [x,y,z] ~:~ xy_0 = x_0 y \}$. 
The Heisenberg group $\HH_n$ acts on $\F^n \times \F^n$ as 
\[
\left( {\begin{array}{c}
	X \\
	Y \\
	\end{array} } \right)
=
\left( {\begin{array}{cc}
	1 & a \\
	0 & 1 \\
	\end{array} } \right)
\left( {\begin{array}{c}
	x \\
	y \\
	\end{array} } \right)
+
\left( {\begin{array}{c}
	b \\
	c \\
	\end{array} } \right)
=
\left( {\begin{array}{c}
	x+ay+b \\
	y+c \\
	\end{array} } \right) \,,
\]
and hence $\Stab((x,y)) = \{ [a,b,0] ~:~ ay+b=0 \}$. 
Further the structure of $\FF{\HH}_n$ is well--known, see, e.g., \cite{Schulte}. 
There are $|\F|^{2n}$ one--dimensional representations which correspond to additive characters for $x,y$, see the group low \eqref{f:Heisenberg_low} and there is a unique nontrivial representation $\pi$  of dimension $|\F|^n$. 
Thus formula \eqref{f:inverse_representations} has the following form
\begin{equation}\label{f:Heisenberg_representations}
f([x,y,z]) = \frac{\d_f (x,y)}{|\F|}+ \frac{\langle \FF{f} (\pi), \pi ([x,y,z]^{-1}) \rangle_{HS}}{|\F|^{n+1}} \,,	
\end{equation}
where $\d_f (x,y) = \sum_z f([x,y,z])$. 
Let us describe the representation $\pi$ in details in the case $n=1$, see, e.g., \cite{Schulte}. 
Let $\zeta = e^{2\pi i/p}$ and $\mathcal{D} = diag(1,\zeta, \dots, \zeta^{p-1})$ and 
\[
W_a  = 
\left( {\begin{array}{ccccc}
	0 & 1 & 0 & \dots & 0  \\
	0 & 0 & 1 & 0 & \dots  \\
	\dots & \dots & \dots & \dots & \dots \\
	0 & \dots & \dots & 0 & 1\\
	a & 0 & \dots & \dots  & 0 \\
	\end{array} } \right) 
\]
be $(p\times p)$ matrix. 
Then $\pi ([x,y,z]) := \zeta^{z+y} \mathcal{D}^y W^x_{\zeta}$. 
The fact that $\pi$ is a representation follows from an easy checkable commutative identity 
\begin{equation}\label{id:Heisenberg_comm}
\zeta^{xy'} \mathcal{D}^{y'} W^x_\zeta = W^x_\zeta \mathcal{D}^{y'} \,.
\end{equation}
Thus there is just one nontrivial representation $\pi$ and a similar situation takes place in the case of the affine group $\Aff (\F)$, see below.

Now we obtain a lemma on products of sets in $\HH_n$.
A similar result was obtained in \cite[Propositions 3--6 and Theorem 1]{HH_MJCNT} but for a special family of sets which are called {\it semi--bricks}.
Given a set $A\subseteq \HH$ we write $K^{-1} (A):= |A|^{-1} \max_{x,y} \d_A (x,y)$.
Hence from the definition of the quantity  $K=K(A)$ one has that for any $x,y\in \F$ the following holds $\d_A (x,y) \le |A|/K$. 

\begin{theorem}
	Let $A,B\subseteq \HH$, $|A| |B| > p^{5}$. 
	Then  $[A,B]$ contains $[0,0,\F_p]$.\\
	Further if $A \subseteq \HH_n$, then  for any $k\ge 2$ and any signs $\eps_j \in \{-1,1\}$ with $\sum_{j=1}^{2k} \eps_j = 0$  
	the product $\prod_{j=1}^{2k} A^{\eps_j}$ contains   $[0,0,\F_p]$, provided 
	\begin{equation}\label{cond:z_Heisenberg}
	|A| > p^{n+1+n/k} \,.
	\end{equation}
	Finally, for $K = K(A)$ and $k\ge 2$, we have 
	\begin{equation}\label{f:z_Heisenberg}
	|A^k| \ge 2^{-1} \min\left\{ Kp, \frac{|A|^k}{p^{(n+1)(k-1)}} \right\} \,.
	\end{equation}	
	\label{t:z_Heisenberg}
\end{theorem} 
\begin{proof}
	We know that for any $a,b \in \HH_n$, $a=[x,y,z]$, $b=[x',y',z']$ their commutator equals $[a;b] = [0,0,xy'-yx']$.
	Hence for any $\la \neq 0$ we must solve the equation $xy'-yx'= \la$, where points  $(x,y)$ and $(x',y')$ are counted with the weights equal  $\d_A$ and $\d_B$. 
	Using Theorem of Vinh \eqref{f:Vinh}, 
	we see that the number solutions to this equation is at least 
	\[
	\frac{|A| |B|}{p} - \| \d_A \|_2 \| \d_B \|_2  \sqrt{p} > 0 \,,
	\] 
	because of our assumption $|A||B| > p^{5}$ and a trivial estimate $\| \d_A \|_2 \| \d_B \|_2 \le p (|A| |B|)^{1/2}$.

	To prove the second part of the theorem  take any $z_*:= [0,0,z] \in [0,0,\F_p]$ and write $S$ for the convolution of $\prod_{j=1}^{2k} A^{\eps_j}$.
	Then  by \eqref{f:convolution_representations}, we have $\FF{S} (\pi) = \prod_{j=1}^{2k} \Cf^{\eps_j} \FF{A} (\pi)$, where $\Cf$ is the conjugation operator.  
	Using \eqref{f:convolution_representations} and  the fact that all one--dimensional representations equal $1$ on $[0,0,\F_p]$, we obtain 
	\[
	r_S (z_*) = \frac{\T^{+}_k (\d_A)}{p} + \frac{\langle \prod_{j=1}^{2k} \Cf^{\eps_j} \FF{A} (\pi), \pi (z_*^{-1}) \rangle_{HS}}{p^{n+1}} 
	\ge 
	\frac{|A|^{2k}}{p^{2n+1}} - \frac{|\langle \prod_{j=1}^{2k-1} \Cf^{\eps_j} \FF{A} (\pi) (\pi), (\pi (z_*^{}) \FF{A}^{\eps_{2k}} (\pi) )^* \rangle_{HS}|}{p^{n+1}} 
	\]
	\begin{equation}\label{tmp:23.06_1}
	\ge
	\frac{|A|^{2k}}{p^{2n+1}} - \| \FF{A}^{} (\pi) \|^{2k-2}_{HS} \cdot \frac{\| \FF{A}^{} (\pi) \|^{2}_{HS}}{p^{n+1}} \ge 
	\frac{|A|^{2k}}{p^{2n+1}} - |A| \cdot \| \FF{A}^{} (\pi) \|^{2k-2}_{HS}  \,. 
	\end{equation}
	Here we have used  the Parseval identity  \eqref{f:Parseval_representations}.
	On the other hand, applying  the Parseval formula again, we get 
	\[
	|A| \ge \frac{\| \FF{A} (\pi) \|^2_{HS}}{p^{n+1}}
	\]
	and hence 
	\[
	\| \FF{A} (\pi) \|^2_{HS} \le |A| p^{n+1} \,.
	\]
	Substituting the last bound into  \eqref{tmp:23.06_1}, we derive
	\[
	r_S (z_*) \ge \frac{|A|^{2k}}{p^{2n+1}} - |A| (|A| p^{n+1})^{k-1} > 0
	\]
	as required.

	To obtain \eqref{f:z_Heisenberg} we use the calculations above and,  applying definition of $\T_k (A)$ from \eqref{def:T_k_ab}, we obtain 
	\[
	\T_k (A) \le p^{-1} \T_k^{+} (\d_A) + |A| \cdot \| \FF{A}^{} (\pi) \|^{2k-2}_{HS}
	\le
	\frac{|A|^{2k}}{pK} + |A| (|A| p^{n+1})^{k-1} = \frac{|A|^{2k}}{pK} + |A|^k p^{(n+1)(k-1)} \,.
	\] 
	Using 
	the Cauchy--Schwarz inequality, we derive $|A|^{2k} \le \T_k (A) |A^k|$
	and hence we complete the proof. 
	$\hfill\Box$
\end{proof}


\begin{remark}	
	A variant of the second part of the lemma above can be obtained for products of different sets and we leave it to the interested reader.  
	Clearly, a lower bound for size of $A$ such that $A^n$ contains $[0,0,\F_p]$  
	is $\Omega_n (p^{n+1})$ 	even in the symmetric case, indeed  just consider all matrices $[0,\F^n_p, P]$, where $|P| < p/(2n)$ is an arithmetic progression. 
	\label{r:A^n}
\end{remark}


Now we need a result from \cite[Lemma 2]{HH_isom}.

\begin{lemma}
	Let $\Gr$ be a group and $X$ be  a maximal subset of $\Gr$ such that 
	\[
	[[a;b];c] = e \,, \quad \mbox{ for any } \quad a,b,c \in X \,.
	\]
	Then $XX=X$. 
	\label{l:3commutator}
\end{lemma}


The first part of Theorem
\ref{t:z_Heisenberg} combined with Lemma \ref{l:3commutator}
imply the following consequence.

\begin{theorem}
	Let $n$ be a positive integer and $\eps \in (0,1/6)$ be any real number. 
	Then there is a finite (nonabelian) group $H$ and a set $A_* \subset H$ with the following properties: \\
	$1)~$ $|A_*| > n$, $|A_* A_*| < 2|A_*|$;\\
	$2)~$ For any $A\subseteq A_*$, $|A| \ge |A_*|^{1-\eps}$ and any finite group $\Gr$ such that   there exists a Freiman $5$--isomorphism from $A$ to $\Gr$, we have $|\Gr| \gg |A|^{1+\frac{1-6\eps}{5}}$. 
	\label{t:HH_isom}
\end{theorem}
\begin{proof}  
	The argument follows the scheme of the proof from \cite{HH_isom}. 
	Let 
	\[
	A_* = \{ [x,y,z] ~:~ x\in \{0,1, \dots, \lceil p^\a \rceil \} ,\, y,z \in \F_p   \}  \subseteq \HH \,, 
	\]
	and  we will choose $\a \in (0,1)$ later. 
	Clearly,  $|A_* A_*| \le 2|A_*| - p^2 < 2|A_*|$.
	Take any $A\subseteq A_*$, $|A| \ge |A_*|^{1-\eps}$ and let  $\rho$ be a   Freiman $5$--isomorphism from $A$ to a group $\Gr$. 
	We can assume that $\Gr = \langle \rho (A) \rangle$ and using Lemma \ref{l:3commutator}, we derive that $\Gr$ is   a two--step nilpotent group. 
	If 
	\begin{equation}\label{cond:1}
	|A| >
	p^{2-2\eps + \a (1-\eps)} \ge p^{5/2} \,,
	\end{equation}
	then by Theorem \ref{t:z_Heisenberg} the set $B:= [A,A] \subseteq AAA^{-1}A^{-1}$ contains $[0,0,\F_p]$.
	One satisfies the last condition taking $\a =  \frac{1+4\eps}{2-2\eps}$.
	We write $g_z \in \Gr$ for $g_z = \rho ([0,0,z])$, $z\in \F_p$.  
	Further by the average arguments one can find $u,v\in \F_p$ and a set $Z \subseteq \F_p$ such that $[u,v,Z] \subseteq A$ and for  $p$ large enough the following holds 
	$|Z| \ge |A|/4p^{1+\a} >1$. 
	Taking two distinct elements $[u,v,i]$, $[u,v,j] \in A$ and putting $h_k = \rho ([u,v,k])$, where $k=i,j$, we form $g_{i-j} := h_j^{-1} h_i \in \Gr$, $g_{i-j} \neq e$. 
	Finally, $\rho(B)$ contains $\rho([0,0,\F_p])$, hence $g_{i-j} \in \rho(B)$ and one can check by induction (see \cite{HH_isom}) that for any $l\ge 1$ the following holds $g_{l(i-j)} = g^l_{i-j}$.   
	In particular, the order of $g_{i-j}$ in $\Gr$ is $p$. 
	Consider Sylow $p$--subgroup of $\Gr$ which we denote by $\Gr_p$. 
	Suppose that  $\Gr_p$ is abelian. 
	We know that $[a;a'] = [0,0,xy'-yx']$ for any $a,a'\in A$ and since $\rho$ is $5$--isomorphism and hence $4$--isomorphism, it follows that $xy'-yx'=0$ on $A$, whence $|A| \le p^2$ and this is a contradiction since $|A| > p^{(2+\a)(1-\eps)}$ and this contradicts with our choice of the parameter $\a$ 
	(see details in \cite{HH_isom}). 
	Otherwise, $\Gr_p$ is nonabelian and  in view of \eqref{cond:1} and our choice of $\a$, we obtain 
	\[
	|\Gr| \ge |\Gr_p| \ge p^3 \gg |A_*|^{3/(2+\a)} \ge |A|^{{3/(2+\a)}} \ge |A|^{1+\frac{1-6\eps}{5}}  
	\]
	as required. 
	$\hfill\Box$
\end{proof}

\bigskip 

It is easy to see from the proof that, although,  possibly, the constant $5$ can be improved
but it is the limit of the method.

\bigskip

Now consider the group of invertible affine transformations $\Aff (\F)$ of a field $\F$, i.e., maps of the form  $x \to ax + b$, $a\in \F^*, b\in \F$ or, in other words, the set of matrices 
\[
(a,b) := \left( {\begin{array}{cc}
	a & b  \\
	0 & 1  \\
	\end{array} } \right)	\,, \quad \quad  a\in \F^*\,, \quad b\in \F  \,.
\]
Here we 
associate with such a matrix the vector $(a,b)$.
Then $\Aff (\F)$ is a semi--product $\F_p^* \ltimes \F_p$ 
with the multiplication  $(a,b) \cdot (c,d) = (ac, ad+b)$.
Clearly, $\Aff (\F)$ acts on $\F$. 
For any $a,a' \in \Aff (\F)$, $a=(x,y)$, $a'=(x',y')$, their commutator equals $[a,a'] = (1,y(1-x')-y'(1-x))$.
The group $\Aff (\F)$ contains the standard unipotent subgroup $U = \{ (1,a) ~:~ a \in \F \}$ as well as the standard dilation subgroup $T = \{(a,0) ~:~ a \in \F^*\}$. 
The centralizer $\mathrm{C} (I)$ of $I$ is $\Aff (\F)$, further, if $g = (x,y) \in U \setminus \{I\}$, then $\mathrm{C} (g) = U$ and otherwise
$\mathrm{C} (g) = \Stab(y(1-x)^{-1})$, where $\Stab (x_0) = \{ (a, x_0 (1-a)) ~:~ a\in \F^* \}$. 
The subgroups $U$ and $T$ are maximal abelian subgroups of $\Aff (\F)$.

There are $(|\F|-1)$ one--dimensional representations which correspond to multiplicative characters of $\F^*$ and because there exist precisely $|\F|$ conjugate classes in  $\Aff(\F)$ we see that there is one more nontrivial representation $\pi$ of dimension $|\F|-1$.   
We have an analogue of  formula \eqref{f:Heisenberg_representations}
\begin{equation}\label{f:Aff_representations}
f((x,y)) = \frac{\d_f (x)}{|\F|}+ \frac{\langle \FF{f} (\pi), \pi ((x,y)^{-1}) \rangle_{HS}}{|\F|} \,,	
\end{equation}
where $\d_f (x) = \sum_y f((x,y))$. 
As above let us describe the representation $\pi$ in details, see, e.g., \cite{Celniker}.
We define $\mathcal{D} = diag (1,\zeta^\o, \dots, \zeta^{\o^{p-2}})$, where $\o$ is any primitive root in $\F_p^*$. 
Then $\pi ((x,y)) := \mathcal{D}^y W^{\ind (x)}_1$ (now $W_1$ is $(p-1)\times (p-1)$ matrix).  
An analogue of identity \eqref{id:Heisenberg_comm} is  
\begin{equation}\label{id:Aff_comm}
W^{\ind (a)}_1 \mathcal{D}^{d} = \mathcal{D}^{ad}  W^{\ind (a)}_1  \,.
\end{equation}
Hence as in the case of the Heisenberg group there is just one nontrivial representation $\pi$ of large dimension  and thanks to this similarity we can consider these two groups together. 
Underline it one more time that the trivial representations of  $\HH$ correspond to additive characters but the trivial representations of  $\Aff$ correspond to multiplicative ones.

\bigskip

Put $K^{-1} (A):= |A|^{-1} \max_{x} \d_A (x)$.
Using the same method as in the proof of  Theorem  \ref{t:z_Heisenberg},  one has 

\begin{corollary}
	Let $A\subseteq \Aff (\F_p)$, $|A| > p^{3/2}$. 
	Then $[A,A]$ contains $(1,\F_p)$.\\ 
	Further for any $k\ge 2$ and any signs $\eps_j \in \{-1,1\}$ with $\sum_{j=1}^{2k} \eps_j = 0$  
	the product $\prod_{j=1}^{2k} A^{\eps_j}$ contains  $(1,\F_p)$,   provided 
	\begin{equation}\label{cond:z_Aff}
	|A| > p^{1+1/k} \,.
	\end{equation}
	For $K = K(A)$ and $k\ge 2$, we have 
	\begin{equation}\label{f:z_Aff}
	|A^k| \ge 2^{-1} \min\left\{ Kp, \frac{|A|^k}{p^{k-1}} \right\} \,.
	\end{equation}	
	\label{c:z_Aff}
\end{corollary} 


As in Remark 	\ref{r:A^n} a lower bound for size of $A$ such that $A^n$ contains $(1,\F_p)$ is $\Omega (p)$  because one can consider the set of all matrices $(\F_p, 0)$ as an example.

Let us demonstrate just one particular usage of Corollary \ref{c:z_Aff}. 

\begin{example}
	Let $A=\{ (a,b) ~:~ a\in \F^*_p,\, b\in \F_p \} \subseteq \Aff(\F_p)$ and $|A| >p^{3/4}$.  
	Then considering $A^{-1}AA^{-1}A$, we see that for any $\la \in  \F_p$ there are $a_i,b_i,c_i \in A$ such that 
	\[
	a(b_1-d_1) + c_1 (b-d) = \la c c_1 = \la a a_1 \,.
	\]
\end{example}

\section{On products of bricks in the Heisenberg group and in the  affine group}
\label{sec:proof}


Now let us obtain an upper bound for the energy of {\it bricks} in $\HH$, see the definition in Theorem \ref{t:energy_bricks_Heisenberg} below. 
In particular, it gives a lower bound for size of the product set of such sets.

\begin{theorem}
	Let $\mathcal{A} = \{ [x,y,z] ~:~ x\in X,\, y\in Y,\, z\in Z \} \subseteq \HH$ be a set. 
	Put $M= \max\{|X|, |Y| \}$.
	Then
	\begin{equation}\label{f:energy_bricks_Heisenberg}
	\E (\mathcal{A}) \lesssim \frac{\E^{+}(Z) |X|^3 |Y|^3}{p} + \E^{+} (Z) |X| |Y| (|X||Y| M^{1/2} + M^2) + \mathcal{E} \,,
	\end{equation}
	where $\mathcal{E}$ is
	\begin{equation*}\label{f:energy_bricks_Heisenberg2}
	\min\left\{\frac{|Z|^4 \E^{+} (X) \E^{+} (Y)}{p} + |Z|^4 |X|^{1/4} |Y|^{9/4} \E^{+} (X)^{3/4}, \frac{|X|^3 |Y|^3 |Z|^4}{p} + (|X||Y|)^{5/2} |Z|^2 \E^{+} (Z)^{1/2} \right\}
	\,.
	\end{equation*}
	\label{t:energy_bricks_Heisenberg}
\end{theorem}
\begin{proof}
	The energy $\E (\mathcal{A})$ equals the number of the solutions to the system
	\begin{equation}\label{f:system_H}
	x+x_* = x'+x'_*,  \quad y+y_* = y'+y'_*, \quad z+z_* + xy_* = z'+z'_* + x'y'_* \,, 
	\end{equation}
	where $x,x',x_*,x'_* \in X$,  $y,y',y_*,y'_* \in Y$,  $z,z',z_*,z'_* \in Z$.
	First of all we consider solutions to \eqref{f:system_H}  with all possible $z,z',z_*,z'_* \in Z$ such that $z + z_* \neq z' + z'_*$.
	Denote by $\sigma_1$ the correspondent  number of the solutions. 
	Then the last equation of our system \eqref{f:system_H} determines a line such that $(x,x') \in X\times X$ and  $(y_*,y'_*) \in Y\times Y$ are counted with the weights $r_{X-X} (x-x')$ and $r_{Y-Y} (y_* - y'_*)$, correspondingly. 
	Clearly, such weights do not exceed $|X|$ and $|Y|$, respectively. 
	Moreover, $\sum_{x,x'\in X} r_{X-X} (x-x') = \E^{+} (X)$ and similar  $\sum_{y,y'\in Y} r_{Y-Y} (y-y') = \E^{+} (Y)$ .   
	Using the pigeonhole principle and applying Theorem \ref{t:line/point_as}, we find a number $0< \D \le |X|$ and a set of lines $L \subseteq X\times X$,
	$\D |L| \le \E^{+} (X)$ such that
	\[
	\sigma_1 \lesssim \frac{|Z|^4 \E^{+} (X) \E^{+} (Y)}{p} + |Z|^4 |Y| \cdot \D |Y|^{5/4} |L|^{3/4} 
	\le 
	\]
	\begin{equation}\label{f:sigma1}
	\le
	\frac{|Z|^4 \E^{+} (X) \E^{+} (Y)}{p} + |Z|^4 |X|^{1/4} |Y|^{9/4} (\E^{+} (X))^{3/4} \,.
	\end{equation}
	Let us give another estimate for $\sigma_1$. 
	Now we crudely bound $r_{X-X} (x-x')$ and $r_{Y-Y} (y_* - y'_*)$ as $|X|$ and $|Y|$, respectively, but treat our equation 
	$z+z_* + xy_* = z'+z'_* + x'y'_*$ as $s + xy_* = s' + x'y'_*$, where $s,s'$ are counted with weights $r_{Z+Z} (s)$, $r_{Z+Z} (s')$. 
	Applying Theorem \ref{t:Misha+} and using the same calculations as above, we obtain 
	\begin{equation}\label{f:sigma1'}
	\sigma_1 \lesssim \frac{|X|^3 |Y|^3 |Z|^4}{p} + (|X||Y|)^{5/2} |Z|^2 \E^{+} (Z)^{1/2}
	\end{equation}
	as required.

	Now consider the remaining case when $z + z_* = z' + z'_*$ and denote the rest by $\sigma_2 / \E^{+} (Z)$. 
	One can check that zero solutions in the remaining variables $x$, $x'$, $y'_*$, $y_*$ as well as solutions with $\a := x/x' = y'_* / y_* = 1$ coins at most 
	\begin{equation}\label{f:interm_est}
	3|X|^2 |Y|^2 + |X| \E^{+} (Y) + |Y| \E^{+} (X) \ll |X| |Y| (|X|^2  + |Y|^2) 
	\end{equation}
	in $\sigma_2$. 
	Thus suppose that $\a \neq 1$ and all variables $x$, $x'$, $y'_*$, $y_*$  do not vanish. 
	We have
	\begin{equation}\label{f:interm}
	(\a-1) x' = x'_*- x_* \,, \quad (\a-1) y_* = y- y' \,.
	\end{equation}
	In particular, $\frac{x'}{y_*} = \frac{x'_* - x_*}{y-y'}$ and if we determine all variables $x',x_*,x'_*, y,y',y_*$ from the last equation, then from \eqref{f:interm}, we know $\a$ and hence 
	recalling $\a = x/x' = y'_* / y_*$, we find the remaining variables $x,y'_*$.  
	Hence 
	\begin{equation}\label{tmp:25.06_1}
	\sigma_2 \le \sum_w r_{X/Y} (w) r_{(X-X)/(Y-Y)} (w) \,.
	\end{equation}
	Using Theorem \ref{t:Misha+}, we get 
	\[
	\sigma_2 \ll \frac{|X|^3 |Y|^3}{p} + |X| |Y| (|X|^2 + |Y|^2) + |X|^2 |Y|^2 (|X|^{1/2} + |Y|^{1/2}) \,.
	\]
	Combining the last estimate, bounds  \eqref{f:sigma1}, \eqref{f:sigma1'} and  \eqref{f:interm_est}, we obtain the required result. 
	$\hfill\Box$
\end{proof}

\bigskip

For example, if $|X|=|Y|=|Z| \le p^{2/3}$, then the result above gives us $\E(\mathcal{A}) \lesssim |\mathcal{A}|^{3-1/6}$.

\bigskip 

Now if $Z=\{0\}$, then we do not need to consider the first case in the proof of Theorem \ref{t:energy_bricks_Heisenberg}, hence $\mathcal{E} = 0$ and whence, we obtain a consequence which is better than \cite[Theorem 2.4]{HH_sum-prod_HH}.

\begin{corollary}
	Let $\mathcal{A} = \{ [x,y,0] ~:~ x,y\in A \} \subseteq \HH$ be a set. 
	Then
	\begin{equation}\label{f:energy_bricks_Heisenberg_HH}
	|\mathcal{A}^2| \gg \min\{ |\mathcal{A}|^{7/4}, p |\mathcal{A}| \} \,.
	\end{equation}
	\label{cor:energy_bricks_Heisenberg_HH} 
\end{corollary}

\begin{remark}
	It was proved in \cite{RS_SL2} that the quantity from \eqref{tmp:25.06_1} can be estimated better for real sets $A\subset \R$, namely, as $O(|A|^{9/2-c})$, where $c>0$ is an absolute constant.
	Hence in $\R$ lower  bound \eqref{f:energy_bricks_Heisenberg_HH} in Corollary  \ref{cor:energy_bricks_Heisenberg_HH}  is even better.  
\end{remark}


We say that two series of  sets $X_i \subseteq \F_p$, $Y_i \subseteq \F_p$ have comparable sizes if for all $i,j \in [n]$ the following holds $|X_i| \ll |X_j|$, $|Y_i| \ll |Y_j|$.
In this case put $\mathcal{X}=\max_{i\in [n]} |X_i|$, $\mathcal{Y}=\max_{i\in [n]} |Y_i|$.

Now we are ready to improve Theorem \ref{t:HH_bricks_Hn} from the Introduction in the situation when $X_i,Y_i$ have comparable sizes. 
It is easy to show that in our result $\eps(n) = \la^n$ for a certain $\la <1$ but in Theorem \ref{t:HH_bricks_Hn} it is just $\eps(n) = O(1/n)$. 
Also, the dependence on $|Z|$ in Theorem \ref{t:bricks_Hn} is better.
Finally, we remark that of course the lower bound $|\mathcal{A}|/p$ for the number of cosets  is optimal.

\begin{theorem}
	Let $n\ge 2$ be an even number, and $X_i,Y_i, Z\subseteq \F_p$, $i\in [n]$,  
	$X = \prod_{i=1}^n X_i \subseteq \F^n_p$, $Y = \prod_{i=1}^n Y_i \subseteq \F^n_p$,
	\[\mathcal{A} = \{ [x,y,z] ~:~ x\in X,\, y\in Y,\, z\in Z \} \subseteq \HH_n\] 
	be sets and $X_i$, $Y_i$ have comparable sizes.
	If $|Z| \le \mathcal{X} \mathcal{Y}$, $\mathcal{X} \le |Z| \mathcal{Y}$, $\mathcal{Y} \le |Z| \mathcal{X}$ and   
	\begin{equation}\label{f:bricks_Hn}
	\mathcal{X} \mathcal{Y} \gtrsim p^{3/2} \cdot \left(\frac{\mathcal{X} \mathcal{Y}}{p|Z|^{1/2}} \right)^{2^{-{n/2}}} \,,
	\end{equation}
	then $\mathcal{A}^2$ contains at least  
	$|\mathcal{A}|/p$ cosets of 
	$[0,0,\F_p]$. 
	\label{t:bricks_Hn}
\end{theorem}
\begin{proof}
	It is enough to prove that $\mathcal{A}^2$ contains $[0,0,\F_p]$,  provided 
	\begin{equation}\label{f:bricks_Hn_inside}
	\mathcal{X} \mathcal{Y} \gtrsim p \cdot \left(\frac{\mathcal{X} \mathcal{Y}}{|Z|} \right)^{2^{-{n/2}}} \,,
	\end{equation}	
	because then \eqref{f:bricks_Hn} 
	follows by arguments from \cite[Theorem 1.3]{HH_Hn}. 
	Indeed, if we replace $X_i$, $Y_i$ by $\t{X}_i := X_i \cap (a_i - X_i)$,  $\t{Y}_i := Y_i \cap (b_i - Y_i)$ for some $a_i,b_i$ and consider  
	the correspondent set $\t{\mathcal{A}}$, then by the group low \eqref{f:Heisenberg_low} the inclusion  $[0,0,\F_p] \subseteq \t{\mathcal{A}}^2$ implies  
	$[\vec{a},\vec{b},\F_p] \subseteq \mathcal{A}^2$, where $\vec{a} = (a_1, \dots, a_n)$, $\vec{b} = (b_1, \dots, b_n)$.
	Further notice that the set $\Omega_i = \{ a ~:~ |X_i \cap (a - X_i)| \ge \zeta  |X_i|^2 /p\}$ has size $|\Omega_i| \ge (1-\zeta) |X_i|$
	and hence taking $\zeta$ such that $(1-\zeta)^{2n} \ge 1/2$ we can find at least 
	\begin{equation}\label{tmp:04.07_1}
	(1-\zeta)^{2n} \prod_{i=1}^n |X_i| \ge |\mathcal{A}|/(2|Z|) \ge |\mathcal{A}|/p
	\end{equation}
	vectors $\vec{a}$, $\vec{b}$ with $|\t{X}_i|\ge \zeta  |X_i|^2 /p$, $|\t{Y}_i|\ge \zeta  |Y_i|^2 /p$. 
	To get \eqref{tmp:04.07_1} 
	we have used the fact that $|Z|< p/2$ because otherwise Theorem \ref{t:bricks_Hn} is trivial. 
	Substitution $\t{X}_i$, $\t{Y}_i$ into \eqref{f:bricks_Hn_inside} gives the desired condition \eqref{f:bricks_Hn}.

	Now let us obtain \eqref{f:bricks_Hn_inside}.
	Take $[x,y,z], [x',y',z'] \in  \mathcal{A}$ and by the group low \eqref{f:Heisenberg_low}  we need to solve the equation 
	\[
	z + z' + x_1 y'_1 + \dots + x_n y'_n = \la \,, \quad \quad z,z'\in Z\,, \quad x_i, x'_i \in X_i \,, \quad  y_i, y'_i \in Y_i
	\]
	for any $\la$. 
	We consider even $n$ only  (recall that we assume that $n\ge 2$) and denote by $\sigma_{n/2}$ the number of the solutions to the last equation. 
	Almost repeating the proof of \cite[Theorem 32]{sh_as} (also, see \cite[Remark 33]{sh_as}), one obtains  an asymptotic formula for $\sigma_k$, namely,
	\begin{equation}\label{tmp:03.07_1}
	\sigma_k - \frac{|Z|^2 |X| |Y|}{p} \lesssim  |Z|^{2-2^{-k}} (\mathcal{X} \mathcal{Y})^{2k-1 + 2^{-k}} \,.
	\end{equation} 
	Indeed, by Theorem \ref{t:Misha+} we know (thanks to $|Z| \le \mathcal{X} \mathcal{Y}$, $\mathcal{X} \le |Z| \mathcal{Y}$, $\mathcal{Y} \le |Z| \mathcal{X}$) that  
	\[
	\sigma_1 - \frac{|Z|^2 |X| |Y|}{p} \ll (\mathcal{X} \mathcal{Y} |Z|)^{3/2}
	\]
	and that the recurrent formula for the error term $\mathcal{E}_k$ in the right--hand side of  \eqref{tmp:03.07_1} is 
	\[
	\mathcal{E}_{k+1} \ll (\mathcal{X} \mathcal{Y})^{3/2} \mathcal{E}^{1/2}_k \cdot |Z| (\mathcal{X} \mathcal{Y})^k \,.
	\]
	Again we need to use our conditions  $|Z| \le \mathcal{X} \mathcal{Y}$, $\mathcal{X} \le |Z| \mathcal{Y}$, $\mathcal{Y} \le |Z| \mathcal{X}$ and 
	induction 
	similar to the proof of \cite[Theorem 32]{sh_as}. 
	Thus asymptotic formula \eqref{tmp:03.07_1} takes place  and $\sigma_k$ is positive  if 
	\[
	\mathcal{X} \mathcal{Y} \gtrsim  p (\mathcal{X} \mathcal{Y} |Z|^{-1})^{2^{-k}} \,.
	\]	
	This completes the proof.
	$\hfill\Box$
\end{proof}

\bigskip

Let us compare condition \eqref{cond:z_Heisenberg} of Theorem \ref{t:z_Heisenberg} (condition (\ref{cond:z_Aff}) of Corollary \ref{c:z_Aff}) and Theorem \ref{t:bricks_Hn} namely, formula \eqref{f:bricks_Hn_inside} from the proof. Theorem \ref{t:z_Heisenberg} concerns general sets but condition  \eqref{f:bricks_Hn_inside} is exponentially better than \eqref{cond:z_Heisenberg}.  
For concrete families of sets one can prove similar exponentially small bounds.
Consider, for example, a brick $\mathcal{A} = \{(a,b) ~:~ a,b\in A\} \subseteq \Aff (\F_p)$, $A\subseteq \F^*_p$ 
and give the sketch of the proof of the existence of this decay (see details in \cite[Remark 34]{sh_as} and in \cite[Theorem 11]{RS_SL2}). 
Put $\mathcal{A}^n = \{ (a_n, b_n) \} $ and by the group low we know that $a_{n+1} = a_n a$, $b_{n+1} = a_n b + b_{n}$, where $a,b\in A$.  
Using the last recursive formula and the arguments as in \cite[Theorem 32]{sh_as} to solve the equation $a_n b + b_{n} = a'_n b' + b'_{n}$, 
we obtain in $\R$ (but similar in $\F_p$) that for any $\mathcal{B}$ from the affine group one has $|\mathcal{B}\mathcal{A}| \gg |A|^{3/2} |\mathcal{B}|^{1/2}$ and this implies the exponential decay. 





\section{Concluding remarks}
\label{sec:concluding}


In 
this section we discuss some further connections between the sum--product phenomenon and growth in the Heisenberg group.

In Theorem \ref{t:energy_bricks_Heisenberg} we have deal with the term  $\sigma_2$. 
It is easy to see that this quantity is  just 
$
\sum_{\la, \mu} \E^\times (A^{+}_\la, A^{+}_\mu) \,,
$
where $A^{+}_\la = A \cap (\la - A)$. 
Hence we have estimated this expression as well. 
In a dual way one can consider
$
\sum_{\la, \mu} \E^+ (A^{\times}_\la, A^{\times}_\mu) \,,
$
where $A^{\times}_\la = A \cap \la A^{-1}$ or, similarly,  $A^{\times}_\la = A \cap \la A$.
Then we have the correspondent analogue of system \eqref{f:system_H}, namely, 
\[
a a_1 = a' a'_1 \,, \quad b b_1 = b' b'_1 \,, \quad a+b_1 = a'+b'_1 \,.
\] 
It gives $b_1 b/b' - b_1 = a' a'_1/a_1 - a'$ (and the remaining variables $a,b'_1$ can be find uniquely) and hence again this can be bounded as $|A|^6/p + O(|A|^{9/2})$ in $\F_p$ and as $O(|A|^{9/2-c})$ in $\R$, where $c>0$ is an absolute constant, see \cite{RS_SL2}.   

In a similar way, one can consider the problem of estimating the quantities 
\begin{equation}\label{f:mixed_sum-prod}
\sum_{\la} \E^{\times} (A^{+}_\la) \,, \quad \quad \sum_{\la} \E^{+} (A^{\times}_\la) \,.
\end{equation}
The first one naturally appears in sum--product questions in $\R$ which are connected with Solymosi's argument \cite{soly}, see, e.g., \cite{KS1}.
As in Theorem \ref{t:energy_bricks_Heisenberg}, we see that the first sum equals the number solutions to the system 
\[
a+a_1 = a'+a'_1=b+b_1=b'+b'_1 \,, \quad \quad ab_1 = a'b'_1 
\]
hence as above $\frac{a'}{b_1} = \frac{a'_1-a_1}{b-b'}$ and
after some calculations we arrive to 
\begin{equation}\label{f:5_var}
(b+b_1-a'_1) (b-b') = b_1(a'_1-a_1) \,.
\end{equation}
Now we can estimate the number solutions to the last equation rather roughly.
Indeed, if we fix a variable, say, $b_1$, then relatively to $a_1,b'$ we have an equation of a line. 
Hence the Szemer\'{e}di--Trotter Theorem \cite{ST} gives us 
$\sum_{\la} \E^{\times} (A^{+}_\la) \ll |A|^{11/3}$
and similar in $\F_p$ via Theorem \ref{t:line/point_as}. 
One can estimate the number solutions to \eqref{f:5_var} further  via the Cauchy--Schwarz and different energies.

As for the  dual question,  it is easy to see that $\sum_{\la} \E^{+}  (A \cap \la A^{-1}) \le |A| \E^{+} (A^{-1})$ and 
$\sum_{\la} \E^{+}  (A \cap \la A^{}) \le |A| \E^{+} (A)$ because the map $(x,y,z,w,\la)_{x+y=z+w} \to \la^{-1} (x,y,z,w)_{x+y=z+w}$ has at most $|A|$ preimages. 
Thus in this case nothing interesting happens and one needs a deeper technique to estimate the sum.

\bigskip

{\bf Problem.} Estimate the sum--product quantities \eqref{f:mixed_sum-prod} in $\R$ and in $\F_p$ (for small $A$). 
We suppose that the correct bound is $O(|A|^{3+\eps})$ for an  arbitrary $\eps>0$.

\bigskip

\noindent{I.D.~Shkredov\\
	Steklov Mathematical Institute,\\
	ul. Gubkina, 8, Moscow, Russia, 119991}
\\
and
\\
IITP RAS,  \\
Bolshoy Karetny per. 19, Moscow, Russia, 127994\\
and 
\\
MIPT, \\ 
Institutskii per. 9, Dolgoprudnii, Russia, 141701\\
{\tt ilya.shkredov@gmail.com}

\end{document}